\documentclass[12pt,reqno]{amsart}
\usepackage{amssymb,mathrsfs,centernot,color}
\usepackage{hyperref}

\newcommand{\nw}{\newcommand}

\nw{\Blue}[1]{{\color{blue} #1}}
\nw{\Red}[1]{{\color{red} #1}}

\nw{\B}{\mathcal{B}}
\nw{\D}{\mathcal{D}}
\nw{\E}{\mathcal{E}}
\nw{\F}{\mathcal{F}}
\nw{\T}{\mathcal{T}}
\nw{\A}{\mathcal{A}}
\nw{\M}{\mathcal{M}}
\nw{\X}{\mathcal{X}}
\nw{\Y}{\mathcal{Y}}
\nw{\cL}{\mathcal{L}}
\nw{\cS}{\mathcal{S}}
\nw{\cH}{\mathcal{H}}
\nw{\cP}{\mathcal{P}}
\nw{\cQ}{\mathcal{Q}}
\nw{\cN}{\mathcal{N}}
\nw{\cZ}{\mathcal{Z}}

\nw{\C}{\mathbb{C}}
\nw{\N}{\mathbb{N}}
\nw{\R}{\mathbb{R}}
\nw{\Z}{\mathbb{Z}}
\nw{\Q}{\mathbb{Q}}

\nw{\al}{\alpha}
\nw{\be}{\beta}
\nw{\ga}{\gamma}
\nw{\de}{\delta}
\nw{\e}{\varepsilon}
\nw{\fy}{\varphi}
\nw{\om}{\omega}
\nw{\la}{\lambda}
\nw{\te}{\theta}
\nw{\s}{\sigma}
\nw{\ta}{\tau}
\nw{\ka}{\kappa}
\nw{\x}{\xi}
\nw{\y}{\eta}
\nw{\z}{\zeta}
\nw{\rh}{\rho}
\nw{\ro}{\rho}
\nw{\up}{\upsilon}

\nw{\De}{\Delta}
\nw{\Om}{\Omega}
\nw{\Ga}{\Gamma}
\nw{\La}{\Lambda}

\nw{\p}{\partial}
\nw{\na}{\nabla}
\nw{\Cu}{\bigcup}
\nw{\Ca}{\bigcap}
\nw{\Du}{\bigsqcup}
\nw{\dup}{\sqcup}
\nw{\re}{\mathop{\mathrm{Re}}}
\nw{\im}{\mathop{\mathrm{Im}}}
\nw{\weak}{\operatorname{w-}}
\nw{\weakto}{\rightharpoonup}
\nw{\supp}{\operatorname{supp}}
\nw{\sign}{\operatorname{sign}}
\nw{\dist}{\operatorname{dist}}
\nw{\sech}{\operatorname{sech}}
\nw{\Ker}{\operatorname{Ker}}
\nw{\Span}{\operatorname{span}}
\nw{\Str}{\operatorname{Str}}

\nw{\lec}{\lesssim}
\nw{\gec}{\gtrsim}
\nw{\slec}{\,\lesssim\,}
\nw{\sgec}{\,\gtrsim\,}
\nw{\follows}{\Longleftarrow}
\nw{\etc}{,\ldots,}
\nw{\ale}{\ |\!\!\!\le}
\nw{\alec}{\ |\!\!\!\lec}

\nw{\I}{\infty}
\nw{\da}{\dagger}
\nw{\empt}{\varnothing}

\nw{\ti}{\widetilde}
\nw{\ba}{\overline}
\nw{\ha}{\widehat}
\nw{\U}{\underline}
\nw{\Lim}{\lim\limits}
\nw{\Sum}{\sum\limits}

\nw{\Br}[1]{\left\{#1\right\}}
\nw{\BR}[1]{\left[#1\right]}
\nw{\LR}[1]{{\langle #1 \rangle}}
\nw{\fracd}[2]{\left(\frac{#1}{#2}\right)}
\nw{\tf}{\tfrac}
\nw{\abs}[1]{\left|#1\right|}
\nw{\diff}[1]{{\triangleleft #1}}
\nw{\pa}{\triangleright}

\nw{\oto}[1]{\overset{#1}{\longrightarrow}}
\nw{\com}[1]{\text{ (#1)}}
\nw{\tand}{\ \text{ and }\ }
\nw{\tor}{\ \text{ or }\ }
\nw{\otw}{(\text{otherwise})}
\nw{\IN}[1]{\text{ in }#1}
\nw{\Bdd}[1]{\text{ bounded in }#1}

\nw{\sidenote}[1]{\marginpar[\raggedleft\tiny #1]{\raggedright\tiny #1}}
\nw{\Del}[1]{}
\nw{\CAS}[1]{\begin{cases} #1 \end{cases}}
\nw{\mat}[1]{\begin{pmatrix} #1 \end{pmatrix}}
\nw{\smat}[1]{\left[\begin{smallmatrix} #1 \end{smallmatrix}\right]}
\nw{\EQ}[1]{\begin{equation}\begin{split} #1 \end{split}\end{equation}}
\nw{\EQN}[1]{\begin{equation*}\begin{split} #1 \end{split}\end{equation*}}
\nw{\pt}{&}
\nw{\pr}{\\ &}
\nw{\pq}{\quad}
\nw{\pQ}{\qquad}
\nw{\pQQ}{\qquad\qquad}
\nw{\pn}{}
\nw{\prq}{\\ &\quad}
\nw{\prQ}{\\ &\qquad}
\nw{\prQQ}{\\ &\qquad\qquad}
\nw{\EN}[1]{\begin{enumerate} #1 \end{enumerate}}
\nw{\ENI}[1]{{\begin{enumerate} #1 \end{enumerate}}}
\nw{\ENA}[1]{{\begin{enumerate} #1 \end{enumerate}}}

\numberwithin{equation}{section}

\newtheorem{thm}{Theorem}[section]
\newtheorem{cor}[thm]{Corollary}
\newtheorem{lem}[thm]{Lemma}

\theoremstyle{remark}
\newtheorem{rem}{Remark}
\newtheorem{defn}{Definition}

\nw{\cC}{\mathcal{C}}
\nw{\sC}{\mathscr{C}}
\nw{\bS}{\mathbb{S}}
\nw{\cR}{\mathcal{R}}
\nw{\bA}{\mathbb{A}}
\nw{\diam}{\operatorname{diam}}
\nw{\bB}{\mathbb{B}}
\nw{\sT}{\mathscr{T}}
\nw{\err}{\operatorname{err}}

\begin{document}

\title[Scattering for 4D Zakharov system]{Scattering for the 4D Zakharov system below the ground  state}

\author[T.~Candy]{Timothy Candy}
\address[T.~Candy]{Department of Mathematics and Statistics, University of Otago, PO Box 56, Dunedin 9054, New Zealand}
\email{tim.candy@otago.ac.nz}

\author[K.~Nakanishi]{Kenji Nakanishi}
\address[K.~Nakanishi]{Research Institute for Mathematical Sciences, Kyoto University, Kyoto 606-8502, Japan}
\email{kenji@kurims.kyoto-u.ac.jp}

\thanks{This work was partly supported by the Marsden Fund Council grant 19-UOO-142, managed by the Royal Society Te
Ap\={a}rangi (T.C.) and  JSPS KAKENHI Grant Numbers 22H01132 and 23K22403 (K.Z.).}

\begin{abstract}
For the Zakharov system in four space dimensions, we prove that all solutions inside the potential well of the ground states
are global and scattering in the energy space, with no other restriction such as symmetry.
The proof has already been reduced by \cite{C} to ruling out the existence of a minimal non-scattering solution that is precompact along some trajectory.
This paper carries out the final step in the proof, namely we exclude the possibility of precompact solutions inside the potential well by combining two distinct arguments depending on the motion of trajectory.
\end{abstract}



\maketitle


\section{Introduction}

The Zakharov system is a coupled Schr\"{o}dinger–wave system introduced by Zakharov \cite{Z} as a model for Langmuir turbulence in plasma. The model takes the form
\EQ{\label{eqn:Zak u n}
 \CAS{i\p_t u-\De u = nu, & u(t,x):\R^{1+d}\to\C,\\
 \p_t^2 n/\al^2-\De n = -\De |u|^2, & n(t,x):\R^{1+d}\to\R,}}
where $u$ and $n$ describe the envelope of the electric oscillation and the ion density fluctuation, respectively,  and $\al>0$ is the ion sound speed.
In the subsonic limit $\al\to\I$, the Zakharov system converges to the nonlinear Schr\"odinger equation (NLS):
\EQ{\label{eqn:NLS}
 i\p_t u - \De u = |u|^2 u,}
see for instance \cite{Mas08}. Solutions to the Zakharov system and the NLS equation exhibit similar properties, but the interaction between the two different dispersions relations (Schr\"odinger and wave) in \eqref{eqn:Zak u n} induces various differences and significant technical challenges.

The goal of this paper is to establish that the scattering criterion for the energy-critical NLS obtained in \cite{D, KM} extends to the Zakharov system. For simplicity, fixing $\al=1$ and letting $D:=\sqrt{-\De}$ and $N:=n-iD^{-1}\p_t n$,
we work with the equivalent first order formulation
\EQ{\label{eqn:Zak1}
 \CAS{i\p_t u - \Delta u = \re (N) u, \\
 i\p_t N + D N = D |u|^2,}
}
for which the conserved energy (Hamiltonian) is written as
\EQ{ \label{def EZ}
 \pt E_Z(u,N):=\int_{\R^d}\tf{|\na u|^2-\re(N)|u|^2}{2}+\tf{|N|^2}{4}dx.}
Hence we naturally consider initial data in the energy class
\EQ{
 (u(0), N(0))\in (H^1\times L^2)(\R^d), }
where the H\"older inequality $\int n|u|^2dx \le \|N\|_{L^2}\|u\|_{L^4}^2$
and the Sobolev embedding $H^1(\R^d)\subset L^p(\R^d)$ for $2\le p\le\tf{2d}{d-2}$
indicate that $d=4$ is the energy critical dimension,
while dimensions $d\le 3$ are subcritical, in the same way as the NLS \eqref{eqn:NLS}.

The local wellposedness of the Zakharov equation in the energy space for $d\le 3$, together with global existence for small data, was established in \cite{BC}
as an early application of the $X^{s,b}$ spaces. In $d=4$, local well-posedness was first proved in \cite{BGHN} using a compactness argument, which was improved in \cite{CHN1} to a contraction argument with suitable modification of $X^{s,b}$.

The scattering for small initial data was first proved in \cite{GN1} in the 3D radial case. In the non-radial setting with $d=3$, small data scattering has been obtained in a weighted space \cite{HPS} and in the energy space with angular derivative \cite{GLNW}. It remains an open problem to prove small data scattering in the whole energy space or any translation-invariant space when $d=3$, and in fact in any function space for $d\le 2$.  In contrast, in the energy critical case $d=4$, small-data scattering was obtained in parallel to the local wellposedness theory in \cite{BGHN}.

For larger data, following the Kenig-Merle method \cite{KM}, in $d=3$ the small-data result of \cite{GN1} was extended in \cite{GNW} to the dichotomy that all radial solutions below the ground state either scatter or grow-up (blow-up or become unbounded in the energy norm).  The same dichotomy was proven in \cite{GN2} in the 4D radial case, where the Kenig-Merle argument was used to establish a uniform global Strichartz estimate for the Schr\"odinger equation with a free wave potential. In the energy critical case $d=4$, the initial condition for scattering is given as
\EQ{\label{eqn:pot well}
 E_Z(u(0),N(0)) < E_Z(W,W^2)=E_S(W), \pq \|N(0)\|_{L^2}<\|W^2\|_{L^2},}
where $W(x):=\tf{1}{1+|x|^2/8}$ is the ground state or the Aubin-Talenti solution to NLS
    \EQ{ - \Delta W = W^3 }
and $E_S(f) = \int_{\R^4} \frac{1}{2} |\nabla f|^2 - \frac{1}{4} |f|^4 dx$ denotes the conserved energy for the NLS equation \eqref{eqn:NLS}. The condition that the data must lie in the potential well \eqref{eqn:pot well} ensures that the energy $E_Z$ is coercive, and is the Zakharov version of the NLS condition $E_S(u(0)) < E_S(W)$, $\| \nabla u(0) \|_{L^2} < \| W^2 \|_{L^2}$ appearing in \cite{KM}.

The uniform global Strichartz estimate was extended in \cite{CHN2} to the non-radial case, proving the global wellposedness inside the potential well \eqref{eqn:pot well}.
Further improving the bilinear estimates for dispersive wave components, \cite{C} proceeded to the final step of the Kenig-Merle method in the 4D non-radial case,
proving: If there is any non-scattering solution in the region \eqref{eqn:pot well}, then there is a minimal one, and its trajectory in the energy space is precompact modulo translations in space. We call precompact solutions \emph{critical elements}.

\begin{defn}[Critical Element]\label{defn:crit elements}
A solution $(u,N)\in C(\R;(H^1\times L^2)(\R^4))$ to \eqref{eqn:Zak1} with\footnote{The condition $u\in L^2_{t,loc} W^{\frac{1}{2}, 4}_x$ is only to ensure that the concept of a `solution' to \eqref{eqn:Zak1} is well-defined and otherwise plays no role in what follows.}
 $u\in L^2_{t,loc} W^{\frac{1}{2}, 4}_x(\R \times \R^4)$ is a \emph{critical element} if there exists a map $c:\R \to \R^4$ such that the orbit
\EQ{ \label{precomp}
 \{(u,N)(t,x+c(t)) \mid t\in\R\} \subset (H^1\times L^2)(\R^4)
}
is precompact.
\end{defn}

In this paper, we accomplish the final step in the Kenig-Merle argument, namely the rigidity:  there are no critical elements inside the potential well \eqref{eqn:pot well}. Thus we complete the proof of scattering below the ground state:

\begin{thm}[Scattering below ground state]  \label{thm:main}
Let $(f, g) \in (H^1 \times L^2)(\R^4)$ with
    \EQ{ E_Z(f,g) < E_Z(W,W^2), \pq \|g\|_{L^2}<\|W^2\|_{L^2}.}
There exists a unique global solution $(u, N) \in C(\R, (H^1\times L^2)(\R^4) )$ to \eqref{eqn:Zak1} with data $(u,N)(0)=(f,g)$ satisfying
		$\|  u \|_{L^2_t W^{\frac{1}{2}, 4}_x(\R^{1+4})} < \infty$, and
    \EQ{ \lim_{t\to \pm \infty} \| (e^{it\Delta} u, e^{-itD} N) - (f_{\pm}, g_{\pm}) \|_{H^1\times L^2} = 0 }
for some scattering states $(f_\pm, g_\pm) \in (H^1 \times L^2)(\R^4)$.
\end{thm}

Strictly speaking, according to the above definition, $W\not\in L^2(\R^4)$ is not a critical element, but we may regard it as a limiting object that shows that scattering fails in general.
Our proof of the non-existence of critical elements inside the potential well relies crucially on the fact that $u(t)\in L^2_x$, and it is an interesting remaining question if one may drop it
as in the case of (energy-critical) NLS.  Another remaining question is the blow-up in the region where the energy remains below the ground state, but $\|g\|_{L^2}>\|W^2\|_{L^2}$. The grow-up is known in radial setting \cite{GN2}. Existence of finite-time blow-up solutions for the 4D Zakharov system has been recently proven in \cite{KT}.
While the blow-up in 2D was constructed in \cite{GM} decades ago, the 3D case is still an open problem.

In view of the results in \cite{C}, to prove Theorem \ref{thm:main} it suffices to show rigidity, namely the non-existence of critical elements inside the potential well \eqref{eqn:pot well}. The main difficulty is the lack of control over the spatial center $c(t)$. Indeed, if $c(0)=0$, then the virial argument from \cite{GN2} already rules out the existence of critical elements satisfying \eqref{eqn:pot well}. In the case of the non-radial NLS \cite{D}, or the nonlinear wave equation \cite{KM2},  there exists a center of mass or energy with velocity given by the conserved momentum. Together with the Galilei or Lorentz invariance, this gives control over the trajectory. However the Zakharov system has no such symmetries, and we do not know of any center of mass/energy with conserved velocity, see Remark \ref{rem:schr/wave vs Zak} below for further details. Thus controlling the center $c(t)$ is the major obstruction to applying the method used in the NLS case \cite{DHR}. On the other hand, the lack of good scaling symmetries for the Zakharov equation was heavily exploited in \cite{C} to preclude (or fix) the scaling parameter to leave only the `soliton like' critical elements. This is in strict contrast to the setting of the critical NLS where the main difficulty was controlling the scaling parameter \cite{KM,KV,D}.

To overcome the difficulty of the position parameter $c(t)$, we first preclude sub-sonic or time-like trajectories, namely the case of $|\dot c(t)|\le 1$ (in average for large time),
by refining the virial-variational estimates from the radial case \cite{GN2}. The super-sonic or space-like case $|\dot{c}(t)| \geq 1$ is precluded by a space-time integral estimate on positive source terms of the wave equation that do not radiate at time infinity. Namely we prove that if $w(t)$ is precompact modulo translations, and $\p_t^2 w - \Delta w = D^\alpha \rho$ with $\rho \geq 0 $ and $\alpha \ge 0$ then we have
        \EQ{ \sup_{(t_0, x_0) \in \R^{4+1}} \int_\R \int_{|x|>|t|} \tf{1}{|x|}(1-\tf{|t|}{|x|})^{(d-3)/2} \ro(t-t_0,x-x_0) dxdt \lec  \|\ro\|_{L^\I_tL^1_x}.  }
See Theorem \ref{thm:orth wave} for a more precise statement. The argument excluding space-like trajectories, as well as the above as the estimate, seems new and may be of independent interest. We hope it may prove useful in other problems as well.

\subsection{Outline} In Section \ref{sec:virial}, we exclude the time-like case $|\dot{c}(t)| \le 1$ by deriving virial-variational estimates around a moving center, together with a localization argument to define a suitable center of mass for the solution. In Section \ref{sec:orth wave}, we derive a space-time estimate for non-radiating wave sources. Finally in Section \ref{sec:proof}, we combine the previous two sections to exclude the existence of critical elements inside the potential well \ref{eqn:pot well}, thereby proving Theorem \ref{thm:main}.

\subsection{Notation} We conclude this section with some notation. The inner products are denoted by
\EQ{ \LR{a,b}:=\re(\bar a b), \pq \LR{f|g}:=\int_{\R^d}\LR{f(x),g(x)}dx.}
The Fourier multipliers are denoted by $\fy(D):=\F^{-1}\fy(|\x|)\F$.
$H^s$ and $\dot H^s$ denote the inhomogeneous and homogeneous $L^2$-Sobolev spaces on $\R^d$.
The $L^p(\R^d)$ norm is abbreviated as $\|\cdot\|_p$.
Open balls in $\R^d$ are denoted by $B_R(z):=\{x\in\R^d\mid |x-z|<R\}$.
We fix a smooth cut-off function $\La:\R\to\R$ and rescale it for $R>0$ as
\EQ{ \label{def cut-off}
 \La(t)=\CAS{1 &(|t|\le 1) \\ 0 &(|t|\ge 2)} \pq \La_R(t):=\La(\tf tR).}

\section{The virial argument for time-like trajectories} \label{sec:virial}
In this section, we apply a virial argument inside the potential well \eqref{eqn:pot well} to rule out the existence of time-like trajectories.
The conclusion is
\begin{thm}[No time-like critical elements]\label{thm:no time-like}
Let $(u, N)$ be a critical element with trajectory $c:\R\to\R^4$ such that the property \eqref{eqn:pot well} holds. There exist $\e_0, T_0>0$ such that for any $t_0,t_1\in\R$ satisfying $|t_0-t_1|\ge T_0$, we have
\EQ{ \label{Lip est on c}
  1 + \e_0 \le \tf{|c(t_0) - c(t_1)|}{|t_0-t_1|} \le \e_0^{-1}. }
\end{thm}
The proof of Theorem \ref{thm:no time-like} relies on a virial argument. As this involves spatial weights, to ensure all quantities are well-defined in the energy space, we are eventually  forced to work with localized virial quantities. As non-local operators are technically inconvenient when using spatial localization, instead of the formulation \eqref{eqn:Zak1}, in this section we work with the Zakharov system in the form
\EQ{ \label{eqn:Zak b}
 \CAS{i\p_t u - \Delta u =  - u \na\cdot b,\\
  \p_t^2 b - \Delta b = \nabla |u|^2, \qquad \p_jb_k = \p_kb_j\ (j,k=1\etc d).}
}
which follows by taking $ n = - \na\cdot b$ and $b = D^{-2} \nabla n$. Clearly, if $(u, N)$ is a critical element, then we obtain a non-trivial solution $(u, b)$ to \eqref{eqn:Zak b} such that  the orbit
$\{ (u, b, \dot{b})(t,x+c(t))) \mid t\in \R \} \subset H^1 \times \dot{H}^1 \times L^2$ is precompact.
We also call $(u,b)$ a critical element.

\subsection{The virial identity}
Conservation of energy and momentum may be written in the divergence form for the Zakharov system \eqref{eqn:Zak b} as follows. For $(u,b)(t,x):\R^{1+d}\to\C\times\R^d$, define the Schr\"odinger mass/momentum/energy density as
    \EQ{  m(u):=\tf 12|u|^2, \qquad  p_S(u):=\LR{u,i\na u}, \qquad e_S(u):=\tf12|\na u|^2-\tf14|u|^4,}
the wave momentum/energy density
    \EQ{ p_W(b):=-\sum_{j=1\etc d}\LR{ \dot b_j,\na b_j}, \qquad e_W(b):=\sum_{j=1\etc d}\tf12(|\dot b_j|^2+|\na b_j|^2), }
and the Zakharov momentum/energy density
    \EQ{
        p_Z(u,b)&:=p_S(u)+p_W(b),\\
        e_Z(u,b)&:=\tf12[|\na u|^2+e_W(b)+|u|^2\na\cdot b] = e_S(u) + \tf14[|\dot b|^2+|\na\cdot b+|u|^2|^2].
    }
We define the corresponding mass, the energy and the momentum for the Zakharov system as
    \EQ{  M(u):=\int_{\R^d} m(u)dx, \quad  E_Z(u,b):=\int_{\R^d}e_z(u,b)dx,\quad P_Z(u,b):=\int_{\R^d}p_Z(u,b)dx,}
where, with a slight abuse of notation, provided that $n=-\na\cdot b$ and $N=n-iD^{-1}\dot n$, we have $E_Z(u,b)=E_Z(u,N)$ as defined in \eqref{def EZ}.
Similarly we define $E_S(u)$, $P_S(u)$ and $P_W(b)$ by integrating the corresponding density over $\R^d$.

If $(u,b)$ satisfies the Zakharov system \eqref{eqn:Zak b}  (in the classical sense) then we have the microscopic conservation laws
    \EQ{ 0 =\p_t m(u) + \na\cdot p_S(u)  &= \p_t p_Z(u,b)+\sum_{k=1}^d \p_k \up_{j,k}(u,b) \\
    &=\p_t e_Z(u,b) + \na\cdot \up_0(u,b)}
where we let
    \EQ{
        \up_0(u,b)&:=-\LR{\na u,\dot u}-\LR{m(u),\dot b}+\tf12 p_W(b),\\
      \up_{j,k}(u,b)&:=2\LR{\p_k u,\p_j u}+2\LR{m(u),\p_jb_k}+\LR{\p_kb,\p_jb} +\de_{j,k}[\tf{|\dot b|^2-|\na b|^2}{2}-\De m(u)].
    }
In  particular, the mass and the Zakharov energy and momentum are all conserved
    \EQ{ 0=\p_t M(u)=\p_t E_Z(u,b)=\p_t P_Z(u,b).}
Since $M$, $E_Z$, and $P_Z$ are continuous functionals in the energy space, their conservation is easily extended to the solutions in the energy class $(u,b,\dot b)\in C_t((H^1\times\dot H^1\times L^2)(\R^d))$  by the local well-posedness for $d\le 4$. On the other hand, the quantities $E_S(u)$, $P_S(u)$ and $P_W(b)$ are not conserved for solutions to \eqref{eqn:Zak b}.

To compute the virial identity for the solution $(u, b)$, we need another equation for the wave component, namely
\EQ{
  (\p_t^2-\De) m(b)=|\dot b|^2-|\na b|^2 - 2\LR{m(u),\na\cdot b} + \na\cdot\LR{2m(u),b}.}
We then define the virial identity for the Zakharov system in terms of the momentum
\EQ{
 \pt V_Z(u,b):=\LR{x|p_Z}-\tf{d-1}{2}\p_tM(b).}
Henceforth the integrals make sense only if the solution has enough decay for $|x|\to\I$. We later introduce a localization argument to extend the virial quantities to solutions in the energy class. A short computation using the above divergence identities gives the classical virial identity
\EQ{
  \dot V_Z(u,b) &=2\|u\|_{\dot{H}^1}^2+\tf12( \| \dot{b} \|_{2}^2 + \| \na\cdot b\|_{2}^2) + \tf{d+1}{2}\LR{  \na\cdot b ||u|^2} \\
                &= 2 K_S(u) + \tfrac{1}{2}( \| \dot b \|_2^2 + \| \na\cdot b + |u|^2 \|_2^2) + \tfrac{d-1}{2} \LR{  \na\cdot b + |u|^2 \mid |u|^2} \\
                &=:2K_Z(u,b)}
where
\EQ{
 K_S(u):=\|\na u\|_2^2-\tf d4\|u\|_4^4 = \LR{E_S'(u)| (x\cdot\na+\tf{d}{2})u},}
arises in the corresponding computation of the Schr\"odinger virial identity.

The above virial identity can be used to show that solutions cannot remain concentrated around the origin $x=0$. In particular, a version of the above played a critical role in the radial setting \cite{GN2}. However, in the non-radial setting, the critical element is potentially moving around some center, and thus we need to adapt the above virial identity to a moving trajectory.
To this end, consider a general (regular) trajectory $y:\R\to\R^d$. We eventually choose $y(t)$ to be a suitable center of mass of our critical element, but for now it is easier to simply work with a general trajectory. Define the translated center and virial as
\EQ{
  & C_S(u):=\LR{x|m(u)}, \quad C_y(u):=C_S(u)-y M(u)=\LR{x-y|m(u)},\\
  & V_y(u,b):=V_Z(u,b) - y P_Z(u,b) = \LR{x-y|p_Z}-\tf{d-1}{2}\p_tM(b).}
A computation shows that
    \EQ{
        \p_tC_y(u)=P_S(u)-\dot y M(u),}
and
    \EQ{
    \p_t V_y(u,b) &= 2K_Z(u,N) -\dot y\cdot P_Z \\
                    &= 2K_Z(e^{i\dot y x/2}u,N) + \dot y\cdot(P_S -\dot y M(u)) - \dot y \cdot P_W.}
In particular we have the identity
\EQ{ \label{eqn:mov vir}
 \p_t[V_y(u,b)-\dot y\cdot C_y(u)]
  =2K_Z(e^{i\dot y x/2}u,N)- \dot y \cdot P_W-\ddot y C_y(u),}
provided the solution has enough decay to define the integrals with weight $x$. These
identities also suggest the convenient definition of $y$ via
\EQ{ \label{eqn:y mot}
 C_y(u) = 0 \quad \iff \quad y=\tf{C_S(u)}{M(u)} \quad \implies \quad  \dot y=\tf{P_S(u)}{M(u)},}
in which case the virial identity \eqref{eqn:mov vir} is simplified to
\EQ{ \label{bd Vy simp}
 \p_t V_y(u,b)
  =2K_Z(e^{i\dot y x/2}u,N)- \dot y \cdot P_W.}

\begin{rem}\label{rem:schr/wave vs Zak}
As a digression, it is worth comparing the identity \eqref{bd Vy simp} with the corresponding identities in the Schr\"odinger and wave settings. Suppose that we have solutions $\phi$ and $\psi$ to
\EQ{
 (i\p_t-\De)\phi=|\phi|^2\phi, \pq (\p_t^2-\De)\psi=0.}
If $\phi$ and $\psi$ are (nice) solutions, then we have
\EQ{
 \pt C_S(\phi):=\LR{x|m(\phi)} \implies \p_t C_S(\phi)=P_S(\phi), \pq \p_t P_S(\phi)=0,
 \pr C_W(\psi):=\LR{x|e_W(\psi)} \implies \p_t C_W(\psi)=P_W(\psi), \pq \p_t P_W(\psi)=0.}
As in \eqref{eqn:y mot}, the corresponding centers of mass/energy  $y_S(t)$ and $y_W(t)$ may be defined, respectively, by
\EQ{
 \LR{x-y_S|m(\phi)}=0, \pq \LR{x-y_W|e_W(\psi)}=0}
which both satisfy the rather nice property that $\dot y_S=P_S/M$ and $\dot y_W=P_W/E_W$ are constant in time. Moreover, for the virial quantities $V_S:=\LR{x|P_S}$ and $V_W:=\LR{x|P_W}-\tf{d-1}{2}\p_tM(\psi)$, and the spatial translated versions $V_{S,y}:=V_S-y_SP_S$ and $V_{W,y}:=V_W-y_WP_W$, we have
\EQ{
 \p_tV_{S,y}=K_S(e^{i\dot y_Sx/2}\phi), \pq \p_tV_{W,y}=E_W-\tf{|P_W|^2}{E_W}.}
In fact, via the Galilei and Lorentz transforms, we may further reduce to the case of zero momentum $P_S=\dot y_S=0$ and $P_W=\dot y_W=0$, i.e. it is possible to reduce a general critical element to one essentially centered at the origin.

In the case of the Zakharov system, there does not seem to be a quantity like $C_S$ or $C_W$ with velocity controlled by $P_Z$.
As we have seen, the equation for $C_S(u)$ is the same as NLS, but $P_S(u)$ is no longer conserved. In contrast, $\p_t C_W(b)$ seems to be ill-defined in the energy space.
Moreover, the Zakharov system \eqref{eqn:Zak b} is not Galilei or Lorentz invariant and thus it seems difficult to reduce the Zakharov virial identity \eqref{bd Vy simp} any further.
\end{rem}

The second term $\dot{y} \cdot P_W$ in the virial identity \eqref{bd Vy simp} is new when compared to the radial setting (where $y=0$ suffices), and is a significant obstruction to extracting any positivity from the right hand side of \eqref{bd Vy simp}. However if $|\dot y|\le 1$ then the extra term $\dot{y} \cdot P_W$ appearing in \eqref{bd Vy simp} may be absorbed by $K_Z$ to prove monotonicity of $V_y$. The required positivity of $K_Z$ is a consequence of a variational bound.

\begin{lem}\label{lem:var}
Let $d=4$ and $\de>0$. If $f \in H^1(\R^4 ; \C)$ and $\nu \in L^2(\R^4 ; \C)$ satisfy
    \EQ{\label{eqn:pot well nu}
        4 E_{S}(f) + \| \nu  \|_2^2 + \delta^2 \le \| W \|_{\dot{H}^1}^2, \qquad \| \nu - |f|^2 \|_2 < \| W \|_4^2
    }
then we have the lower bound with $C_0:=\tf{4-\sqrt{13}}{2}>0$ for all $\x\in\R^4$
\EQ{ \label{est on KZ}
  2 K_S(e^{i\xi \cdot x} f) + \tf12 \| \nu \|_2^2 + \tf32 \LR{ \nu  \mid |f|^2 } \ge C_0\delta \| f \|_{L^4}^2 + (\| \re \nu\|_2 + \||f|^2\|_2) \| \im \nu\|_2.}
\end{lem}
\begin{proof}
The variational estimates \cite[Lemmas 6.1--6.2]{GN2} with \eqref{eqn:pot well nu} and $N:=|f|^2-\nu$ yield
\EQ{ \label{est on f}
   K_S(f) \ge \| f \|_4^2  ( \delta^2 + \| \nu\|_2^2)^{\frac{1}{2}} \ge  \| f \|_4^2  \max(\delta,\| \nu\|_2), }
which is extended to all $\x\in\R^4$
        \EQ{\label{eqn:inf K_S(f)}
         K_S(e^{i\xi \cdot x} f) \ge \| f \|_4^2  \max(\delta,\| \nu\|_2).
        }
Indeed, it is obvious if $K_S(e^{i x \cdot \xi} f) \ge K_S(f)$.
On the other hand,  if $K_S(e^{i\xi \cdot x} f) < K_S(f)$ then $E_S(e^{i\xi \cdot x} f) < E_S(f)$ and thus $\tilde{f} := e^{i \xi \cdot x} f$ satisfies the energy constraint \eqref{eqn:pot well nu}, so we may replace $f$ with $\tilde f$ in \eqref{est on f}.
To conclude the proof, we take $\nu_1 = \re{\nu}$ and $\nu_2 = \im{\nu}$ and observe that
    \EQ{
         \pt \| \nu \|_2^2 + 3 \LR{ \nu  \mid |f|^2 } - 2 (\| \nu_1\|_2 + \||f|^2\|_2) \| \nu_2 \|_2
                \prq\ge \big( \| \nu_1 \|_2 - \| \nu_2 \|_2\big)^2  -  (3\|\nu_1\|_2 + 2\|\nu_2\|_2)\|f\|_4^2.
    }
As $3\|\nu_1\|_2 + 2\|\nu_2\|_2 \le \sqrt{13}\|\nu\|_2$ (by Cauchy-Schwarz on $\R^2$) we then have
    \EQ{ \| \nu \|_2^2 + 3 \LR{ \nu  \mid |f|^2 } \ge 2 \| \nu_1 - |f|^2\|_2 \| \nu_2 \|_2 - \sqrt{13} \| f\|_4^2 \| \nu\|_2. }
Applying \eqref{eqn:inf K_S(f)} with the lower bound $\|f\|_4^2\|\nu\|_2$ for $\tf{\sqrt{13}}{2}K_S$ and $\|f\|_4^2\de$ for the rest,
we obtain the desired estimate.
\end{proof}

Using the lemma with $\nu := \na\cdot b +|u|^2 + i |\dot{b}|$ and $|P_W| \le \| \dot{b} \|_{2} \| \nabla b \|_2 = \| \dot{b} \|_2 \| \na\cdot b \|_2$
in \eqref{bd Vy simp}, we obtain
    \EQ{ \p_t[V_y(u,b)] \ge C_0\delta \| u \|_4^2 + (1-|\dot{y}|) \|  \dot{b} \|_2 \| \na\cdot b\|_2. }
Provided $|\dot{y}|\le 1$, we would then conclude that $V_y(u,b)$ is increasing, which (provided all quantities are well defined!) would exclude the possibility of time-like trajectories. Although there are still some complications to overcome, this observation largely forms the heart of the proof of Theorem \ref{thm:no time-like}.

\subsection{Localized center and virial}
The above virial argument is merely formal, since $C_y$ and $V_y$ are not well defined in the energy space. In particular, to exploit the virial identity \eqref{bd Vy simp} and the center of mass \eqref{eqn:y mot}, we need to localize the integrals to (large) spatial balls.

Let $\chi:\R\to\R$ be odd, smooth, bounded including the derivatives, and strictly increasing with $\chi(x)=x$ for $|x|\le 2$.
Define $\chi_R(x):=\chi(\tf{x}{R})$ for $R>0$ and $x\in\R$, and let $X_R(x):=R(\chi_R(x_1)\etc \chi_R(x_d))$ for $R>0$ and $x\in\R^d$.
We consider $X_R(x)$ with $R\to\I$ as bounded approximation of $x$ in $C_y$ and $V_y$, namely
\EQ{\label{eqn:implicit setup}
 \pt C_y^R(u):=\LR{X_R(x-y)|m(u)},
  \pr V_y^R(u,b):=\LR{X_R(x-y)|p_Z(u,b)}-\tf{d-1}{2}\LR{\La_{2R}(|x-y|) b|\dot b},}
which are continuous functionals on the energy space $(u,b,\dot b)\in(H^1\times \dot H^1\times L^2)(\R^d)$ for $d=3,4$. With these definitions in hand, given $R>0$ we define the trajectory $y(t):\R \to \R^d$ as the solution to
\EQ{\label{eqn:traj}
C_y^R( u) = 0.
}
The condition \eqref{eqn:traj} uniquely defines the trajectory $y\in C^1(\R;\R^d)$. Indeed, if we let $(u,b)$ be any fixed solution in the energy space with $M(u)>0$, and take
    \EQ{\label{eqn:defn f} f(t,s):=\LR{X_R(x_j-s)|m(u(t))},}
then  the existence of $y_j(t)$ solving $f(t,y_j(t))=0$ at any $t\in\R$ follows from the intermediate value theorem, as $f(t,s) \to\mp R\chi(\I)M(u)$ as $s\to\pm\I$. Moreover, by the implicit function theorem, $y_j(t) \in C^1(\R;\R)$ is the unique solution to \eqref{eqn:defn f} since the derivatives
    \EQ{ f_s(t,s) = - \LR{ \chi'_R(x_j - s) | m(u)} <0, \qquad f_t(t,s) = \LR{X_R'(x_j-s)|p_{S,j}(u)} }
are continuous for $(t,s)\in\R^2$.

Thus we obtain the unique solution $y\in C^1(\R;\R^d)$ of \eqref{eqn:traj}
(depending on $u$ and $R$) and moreover
\EQ{\label{eqn:loc mass/mom+e}
 0 = \p_t C_y^R(u) \pt=-\LR{\dot y\cdot\na X_R(x-y)|m(u)}+\LR{\na X_R(x-y)|p_S(u)}
  \pr=P_S(u)-\dot yM(u) + \err_C^R,}
with the error term $\err_C^R$ defined by $\err_C^R=(\err_{C,1}^R\etc\err_{C,d}^R)$ with
\EQ{\label{eqn:error C}
 \err_{C,j}^R=\LR{\chi'(\tf{x_j-y_j}{R})-1|p_{S,j}(u)-\dot y_j m(u)},}
for $j=1\etc d$. Similarly
\EQ{\label{eqn:loc vir+e}
 \pt\p_t V_y^R(u,b) = 2K_Z(u,b)-\dot y\cdot P_Z + \err_V^R = 2K_Z(e^{i\dot y \cdot x/2} u, b) - \dot{y} \cdot P_W(b) + \err_V^R,}
where we have
\EQ{\label{eqn:error V}
\err_V^R \pt:=\sum_{j=1\etc d}\LR{\chi'(\tf{x_j-y_j}{R})-1|\up_{j,j}(u,b)-\dot y_j p_{Z,j}(u,b)}
 \prq+\tf{d-1}{2}\LR{(1-\La_{2R}(|x-y|))||\dot b|^2-|\na b|^2+b\cdot\na |u|^2}
  \prq+\tf{d-1}{2}\sum_{j=1\etc d}\LR{\La'(\tf{|x-y|}{2R})\tf{x-y}{|x-y|}b_j|\dot y\dot b_j+\na b_j},
}
which is also valid for the solutions in the energy space for $d=3,4$.

To make the above error terms small enough, we need to ensure that the critical element $(u, b)$ remains uniformly (in $t\in \R$) concentrated around $y(t)$ defined by \eqref{eqn:traj}. This is not immediately obvious due to the unboundedness of $C_y^R$ as $R\to\I$, and requires measuring the `tightness' of the solution along the trajectory $y(t)$. To this end, given $\fy \in L^1_x(\R^d)$ we define the minimal radius to include most of the mass
(with $\e$-loss)
\EQ{
R(\fy,\e):= \inf\Big\{R>0 \,\,\Big| \,\, \inf_{z\in \R^d} \|\fy\|_{L^1(|x-z|>R)} \le \e\Big\}.
}
We then have the following more general statement.

\begin{lem}\label{lem:cen exist}
Let $\eta  :\R\to\R$ be odd, bounded and strictly increasing. Let $d\in\N$, $0\le \fy\in L^1(\R^d)$ with $\|\fy\|_1>0$.
Then there exists a unique $y:(0,\I)\to\R^d$ such that
\EQ{\label{eqn:defn h}
 \LR{\eta(\tf{x_j-y_j(R)}{R})|\fy}=0 \pq(j=1\etc d),}
Moreover, for any $0<\e<C_\y\|\fy\|_1$ and $R \ge 2(d+1) R(\fy, \e)$ we have
                        \EQ{ \label{eqn:tight 1}
                        \int_{|x-y(R)|>R}\fy(x)dx \le \e,
                        }
where $C_\y:=\eta(\frac{1}{2d})/(\eta( \frac{1}{2d}) + \eta(\infty))>0$.
\end{lem}
\begin{proof}
Let $\ro_j(h):=\LR{\eta(\tf{x_j-h}{R})|\fy}$ for $h\in\R$.
Since $\y$ is strictly increasing, $\fy\ge 0$ and $\|\fy\|_1>0$, $\ro_j(h)$ is strictly decreasing
with $\ro_j(h)\to\mp\eta(\I)\|\fy\|_1$ as $h\to\pm\I$. Hence the intermediate value theorem yields the unique $y_j\in\R$ such that $\ro_j(y_j)=0$.
We now turn to the tightness bound. Let $0<\e<C_\y\|\fy\|_1$ and $R_* := R(\fy, \e)$. Unpacking the definition of $R_*$ gives $z^* \in \R^d$ such that
        \EQ{ \label{eqn:z tight}
        \int_{|x-z^*|> R_*} \fy(x) dx \le \e. }
Since $\eta$ is increasing and odd, we have via the tightness bound \eqref{eqn:z tight}
\begin{align*}
    \rho_j(z^*_j + R_* + (2d)^{-1} R) &\le \eta( \tfrac{-1}{2d}) \int_{|x-z^*|\le R_*} \fy(x)dx + \eta(\infty) \int_{|x-z^*|> R_*} \fy(x) dx \\
    &\le - \eta( \tfrac{1}{2d}) ( \| \fy \|_1 - \e) + \eta(\infty) \e < 0,
\end{align*}
since $\e < C_\y\|\fy\|_1$.
The same argument yields $\rho_j(z^*_j - R_* - (2d)^{-1} R)>0$ thanks to the odd symmetry, and hence by definition of $y(R)$,
\EQ{
 |y_j(R)-z^*_j| < R_*+\tf{R}{2d}}
for all $j=1\etc d$.
Then provided $|x-y(R)| \ge R \ge  2(d+1) R_*$ we see that
    $$ |x-z^*| \ge |x-y(R)| - |z^* - y(R)| \ge R - d R_* - \tf12 R \ge  R_*.$$
Consequently the tightness bound \eqref{eqn:tight 1} follows from \eqref{eqn:z tight}.
\end{proof}

It is tempting to take the limit $R\to\I$ in the definition of $y(t)$ for the critical element,
but it seems in general impossible, as the tightness implies only $y(t) = c(t) + o(R)$  as $R\to \infty$. On the other hand, the solution $(u, b)$ remains precompact along $y(t)$ for each fixed $R>0$. This is not on its own sufficient to deal with the error terms \eqref{eqn:error C} and \eqref{eqn:error V}, as we require control of the error as $R\to \infty$. It is at this step where the uniform nature of Lemma \ref{lem:cen exist} is crucial. In particular, provided that $(u, b)$ is a critical element, we can now conclude via the identities \eqref{eqn:loc mass/mom+e} and \eqref{eqn:loc vir+e} that the modified center $C_y^R$ and virial $V^R_y$ essentially satisfy the same identities as their unmodified counterparts.

\begin{cor}[Localized center/virial]\label{cor:cen}
Assume that $(u, b)$ is a critical element that is precompact along the trajectory $c(t):\R \to \R^d$. There exists $L_0>0$ such that for every $R\ge L_0$ we have a unique solution $y \in C^1(\R;\R^d)$ to \eqref{eqn:traj}. Moreover
$\sup_{t\in\R}|y(t)-c(t)|\le R+L_0$ and
        \EQ{ \label{eqn:loc V/C}
            M(u) \dot{y} &= P_S(u) + o(1), \\
            \p_t V^R_y(u, b)&= 2 K_Z( e^{i \dot{y} \cdot x/2} u, b) - \dot{y} \cdot P_W(b) + o(1),
        }
where the $o(1)$ terms decay to zero as $R\to \infty$ uniformly in time.
\end{cor}
\begin{proof}
The unique existence of $y\in C^1$ has already been checked below \eqref{eqn:traj}.
On the other hand, as $(u,b)$ is precompact along the trajectory $c(t)$, for each $\epsilon>0$ there exists $L(\epsilon)>0$ such that
        \EQ{\label{eqn:ub tight}
            \sup_{t\in \R} \int_{|x-c(t)|>L(\epsilon)} |u|^2 + |\nabla u |^2 + |\dot{b}|^2 + |\nabla b|^2 dx < \epsilon.
        }
In particular, this implies that $L(\epsilon) \ge R(|u(t)|^2, \epsilon)$ for all $t\in\R$. Hence for any $R\ge 2(d+1) L(\epsilon)$, Lemma \ref{lem:cen exist} gives
       \EQ{ \int_{|x-y(t)|\ge R} |u|^2 dx < \epsilon}
and hence we must have
       \EQ{ \int_{B_R(y(t))} |u|^2 + \int_{B_{L(\epsilon)}(c(t))} |u|^2 > 2\|u\|_2^2-2\epsilon.}
This forces $B_R(y(t)) \cap B_{L(\epsilon)}(c(t)) \not = 0$ for $\epsilon\le\|u\|_2^2/2$, which means $|y(t) - c(t)| < R + L(\epsilon)$.
Thus, choosing $L_0\ge 2(d+1)L(M(u))$ ensures $|y(t)-c(t)|<R+L_0$.

To conclude the tightness propery, we observe that the bound \eqref{eqn:ub tight} together with $R>2L(\epsilon)$ implies
        \EQ{\label{eqn:ub tight y}
            \sup_{t\in \R} \int_{|x-y(t)|\ge 2R} |u|^2 + |\nabla u |^2 + |\dot{b}|^2 + |\nabla b|^2 dx < \epsilon.
        }
In other words the solution $(u,b)$ remains concentrated along the trajectory $y(t)$. Moreover, recalling the definitions \eqref{eqn:error C} and \eqref{eqn:error V} we see that
        \EQ{ |\err_C^R| \lec (1+|\dot{y}(t)|) \ \int_{|x-y(t)| \ge 2R} |u|^2 + |\nabla u |^2 dx \lec{}(1+|\dot{y}(t)|) \epsilon}
and, using the Hardy inequality for $b\in\dot H^1(\R^d)$,
        \EQ{ \err_V^R &\lec (1+|\dot{y}(t)|)  \int_{|x-y(t)| \ge 2R} |u|^2 + |\nabla u|^2 + |u|^2 |\nabla b| + |\dot{b}|^2 + |\nabla b|^2dx \\
        &\lec (1+|\dot{y}(t)|)  \epsilon.  }
provided  $d=3,4$. A short computation using the identities \eqref{eqn:loc mass/mom+e} and \eqref{eqn:loc vir+e} then gives the claim \eqref{eqn:loc V/C}.
\end{proof}

\subsection{Proof of Theorem \ref{thm:no time-like}}

We are now ready to give the proof of Theorem \ref{thm:no time-like}.

\begin{proof}[Proof of Theorem \ref{thm:no time-like}]
Given the critical element $(u,N)$, let $b:=D^{-2}\na\re N$, so that $(u,b)$ is a critical element for \eqref{eqn:Zak b}.
Then the energy constraint \eqref{eqn:pot well} allows us to apply Lemma \ref{lem:var} to $(f,\nu)=(u, |u|^2 + \na\cdot b  + i |\dot{b}|)$.
Since $\|u(t)\|_2=\|u(0)\|_2>0$, the precompactness of $u(t,x+c(t))$ in $L^2_x$ implies
 $\inf_{t\in \R} \| u(t) \|_4 >0$. In view of the variational bound \eqref{est on KZ}, we then have $\delta>0$ such that for all $t\in \R$
        \EQ{ 2K_Z(e^{i\dot y\cdot x/2}u, b) -\| \dot{b} \|_2 \| \na\cdot b\|_2  \ge \delta.}
Combining this with the second estimate of \eqref{eqn:loc V/C} (and recalling that $\p_jb_k=\p_kb_j$) we conclude that
        \EQ{ \p_t V^R_y \ge \delta + \| \dot{b} \|_2 \| \nabla b\|_2 - \dot{y} \cdot P_W + o(1). }
Applying a rotation in $\R^4$ if needed, and noting that the Zakharov momentum $P_Z$ is conserved, it is enough to consider the case where $P_Z = P_{Z, 1} \mathbf{e}_1$ points only in the $\mathbf{e}_1$ direction and $P_{Z,1}\ge 0$. Using the first estimate of \eqref{eqn:loc V/C} and taking $y^\bot := y - y_1 \mathbf{e}_1$, we can write
       \EQ{
        \dot y \cdot P_W = \dot{y}_1 P_{W, 1} + \dot y^\bot \cdot ( P_Z - P_S) &= \dot{y}_1 P_{W, 1} - \dot y^\bot \cdot ( M(u) \dot y + o(1)) \\
                            &\le \dot{y}_1 P_{W, 1} + o(1).
        }
Combining the above bounds, and choosing $R>1$ sufficiently large, we obtain
\EQ{ \label{loc vir bd}
 \pn\p_t V_y^R
 \pt\ge  \de+ \| \dot{b} \|_2 \|\nabla b\|_2 - \dot{y}_1 P_{W, 1} +  o(1)
 \pr\ge  \de + \max_\pm(\pm 1-\dot y_1)P_{W,1} + o(1)
 \pr=  \de + \max_\pm(\pm 1-\dot y_1)(P_{Z,1} - \dot{y}_1 M)+ o(1)
 \pr\ge \de/2 + F(\dot y_1),}
where we define $F(s):=\max_\pm(\pm 1-s)(P_{Z,1}-Ms)$. If $P_{Z,1}\le M$ then $F\ge 0$ so $V_y^R(t)$ is (strictly) increasing, which contradicts the bound $|V_y^R| \lec R$. On the other hand if $P_{Z,1}>M$, then we use
\EQ{
 F(s) \pt\ge (1-s)(P_{Z,1}-Ms)
  \pn\ge
    \CAS{-\tf{(P_{Z,1}-M)^2}{4M^2},\\
     (1-s)(P_{Z,1}-M).}  }
The uniform bound for $F(\dot y_1)$ implies monotonicity of $V_y^R$ when $P_{Z,1}-M<M\sqrt{\de}$.
Otherwise, applying the second bound to $F(\dot y_1)$ above, and integrating the resulting inequality in time,
we obtain for any $t_0<t_1$,
\EQ{
  [V_y^R]_{t_0}^{t_1}
  \ge  (\de/2+P_{Z,1}-M)(t_1-t_0)-(P_{Z,1}-M)(y_1(t_1)-y_1(t_0)). }
Since the left side is uniformly bounded by $R$, we conclude:
$P_{Z,1}-M\ge M\sqrt{\de}$ and,
after fixing $R>1$ large enough for \eqref{loc vir bd},
there is $L\sim R/\de$ such that for all $t_0,t_1\in\R$,
\EQ{ \label{TLC}
  t_1-t_0 \ge L \implies 1+\tf{\de}{3(P_{Z,1}-M)} \le \tf{y_1(t_1)-y_1(t_0)}{t_1-t_0},}
 while $\dot y=P_S/M+o(1)$ gives a uniform upper bound. Potentially taking $L$ slightly larger, Corollary \ref{cor:cen} then transfers this bound to the original trajectory $c(t)$.
\end{proof}

\section{Space-time estimate for non-radiating positive sources}\label{sec:orth wave}

In this section, we derive a space-time estimate
to exclude space-like trajectories, exploiting the positivity and precompactness of the source $|u|^2$ in the wave equation. Namely we focus on the linear wave equation
for $\psi:=(-\De)^{-1}n$
\EQ{
  (\p_t^2-\De)\psi=|u|^2,}
ignoring the Schr\"odinger equation for $u$. More precisely, we use only that $|u|^2\ge 0$ belongs to the following function space:
\EQ{ \label{def sM}
 \pt \sT:=\{\ro\in L^\I_t L^1_x(\R^{1+d}) \mid \forall\fy\in\cS_0(\R^d),\ \LR{ e^{\pm i t D} \fy \mid \rho}_{t,x} = 0 \},}
where $\cS_0$ denote the subspace of Schwartz class $\cS$ vanishing at the frequency $0$:
\EQ{
 \cS_0(\R^d) \pt:=\{\fy\in\cS(\R^d) \mid \forall \al\in\N_0^d,\ \LR{x^\al|\fy}=0\}
 \pr=\{\fy\in\cS(\R^d) \mid \forall\al\in\N_0^d,\ \p_\x^\al\F\fy(0)=0\}. }
The restriction to $\cS_0$ is crucial for our purpose, since $N\in C(\R;L^2(\R^4))$ by itself does not allow to define $(-\De)^{-1}n$ in $\cS'(\R^4)$,
because of the singularity at $\x=0$ in the Fourier space.
We could make $\sT$ bigger by replacing $L^1_x$ with the space of bounded Borel measures, without essentially changing the following arguments,
but we stay within $L^1$ to avoid unnecessary distraction.
The goal of this section is
\begin{thm}[Space-time bound on non-radiating sources]\label{thm:orth wave}
For every $d\ge 4$, there is $C_d\in(0,\I)$ such that for any non-negative $\ro\in\sT$ we have for all $(t_0,x_0)\in\R^{1+d}$
\EQ{ \label{ext est}
  \iint_{|x|>|t|} \tf{1}{|x|}(1-\tf{|t|}{|x|})^{(d-3)/2} \ro(t-t_0,x-x_0) dxdt \le C_d\|\ro\|_{L^\I_tL^1_x}. }
\end{thm}
The above bound looks somewhat similar to the Morawetz estimate, but it is for non-radiating source rather than dispersive waves. In fact, the integral is infinite in the case $\ro=|u|^2$ for a free Schr\"odinger wave $u$ with $C_0^\I(\R^d)$ initial data, due to the asymptotic formula $|u|^2\sim t^{-d}|\hat u(0,Cx/t)|^2$ as $t\to\I$.

Obviously, $\sT$ is invariant under space-time translations, spatial rotations, reflections in time, the rescaling $\ro\mapsto \la^{d}\ro(\la t,\la x)$ for $\la>0$,
and the Lorentz transforms (which is easily seen in the space-time Fourier transform).
It is also closed if $d\ge 4$, since the free waves in $\cS$ are included in $L^1_t L^\I_x$.
On the other hand, $\sT$ is non-trivial, containing all time-like traveling sources
\EQ{
 \sT\supset\{\mu(x-ct) \mid \mu\in L^1(\R^d),\ c\in\R^d,\ |c|<1\}, }
since the free waves $w=e^{\pm itD}\fy$ with $\fy\in\cS_0(\R^d)$
may be integrated in time, namely $w=\p_t(\mp ie^{\pm itD}D^{-1}\fy)$
with $D^{-1}\fy\in\cS_0(\R^d)$, hence $\mu(x)\in\sT$
and it is extended to $|c|<1$ by the Lorentz transform.
Note that the left side of \eqref{ext est} becomes equivalent to the right side in this case.
On the other hand, \eqref{ext est} immediately precludes the case of $|c|\ge 1$ for any non-trivial non-negative $\mu\in L^1_x$.

Another subset of $\sT$ is given by the following, which is more relevant for our use.
\begin{lem} \label{lem:transparent}
Let $\ro\in L^\I_tL^1_x(\R^{1+d})$, $\al\ge 0$, $s\in\R$, $c:\R\to\R^d$, and $w\in C(\R;H^s(\R^d))$ satisfy
$(\p_t^2-\De)w = D^\al \ro$ and $\{w(t,x+c(t))\}_{t\in\R}\subset H^s(\R^d)$ is precompact. Then $\ro\in\sT$.
\end{lem}
As the proof shows, $D^\al$ may be generalized to Fourier multipliers $m(D)$ for any $m\in C(\R^d,\C)$ such that $1/m$ is smooth except $0$ with at most polynomial growth at $0$ and at infinity.
\begin{proof}
For any $\fy\in\cS_0(\R^d)$, let $v:=e^{itD}\fy$.
Since $v\in L^1_tL^\I_x$, we have
\EQ{
 \LR{v|\ro}_{t,x}=\lim_{R\to\I}\LR{\La_Rv|\ro}_{t,x},}
where $\La_R$ is the smooth cut-off in \eqref{def cut-off},
and the right side is equal to
\EQ{
 \pt\LR{\La_Rv|\ro}_{t,x}=\LR{\La_R D^{-\al}v|(\p_t^2-\De)w}_{t,x}
 =\LR{(\p_t^2-\De)(\La_R D^{-\al}v)|w}_{t,x}
 \pr=\LR{\La_R''v_0+2\La_R'v_1|\LR{D}^{-s}w}_{t,x}, \pq  v_j:=i^jD^{j-\al}\LR{D}^sv.}
Then $v_j$ are also free waves in $\cS_0$, while $\LR{D}^{-s}w(t,x+c(t))$ is precompact in $L^2_x$.
Since $\|v_j(t)\|_{L^2_x}$ is bounded and $\|v_j(t)\|_{L^\I_x}\to 0$ as $|t|\to\I$,
we deduce that $\LR{v_j|\LR{D}^{-s}w}_x\to 0$ as $|t|\to\I$.
Since $\supp\La_R^{(j)}\subset[-2R,-R]\cup[R,2R]$ and $|\La_R^{(j)}|\lec R^{-j}$ for $j=1,2$,
we conclude that $\LR{\La_Rv|\ro}_{t,x}\to 0$ as $R\to\I$.
\end{proof}

The above lemma produces many examples in $\sT$, such as $(\p_t^2-\De)\fy(x-c(t))$ for any $\fy\in H^s(\R^d)$, $s\in\R$ and any $c\in C^2(\R;\R^d)$.
In particular, there are many super-sonic traveling sources in $\sT$ that are smooth and compactly supported.
However, Theorem \ref{thm:orth wave} implies that those cannot be non-negative if $c(t)$ is space-like (for long time).

The proof of Theorem \ref{thm:orth wave} follows by simply testing against a rather singular self similar free wave.
\begin{proof}[Proof of Theorem \ref{thm:orth wave}]
Let $w_*$ be the singular free wave
\EQ{
 w_*(t,x) := D^{1-d}\cos(tD) \de = \tf{c_d}{|x|}(1-\tf{t^2}{|x|^2})_+^{(d-3)/2} \in L^{1,\I}_t L^\I_x \cap L^\I_t L^{d,\I}_x,}
where $c_d>0$ is some constant, and $L^{p,\I}$ denotes the Lorentz (or weak Lebesgue) space.
Note that $w_*(t,x)=\la w_*(\la t,\la x)$ for all $\la>0$.
Let  $P_j$ be a standard Littlewood-Paley decomposition to the frequencies $|\x|\sim 2^j$ for $j\in\Z$, and $P_{\ge j}:=\sum_{k\ge j}P_k$, and $P_{<j}:=\sum_{k<j} P_k$, for $j\in\Z$. Since $P_jw_*=2^j P_0w_*(2^jt,2^jx)$ is a free wave in $\cS_0(\R^d)$, the dispersive decay yields for $d\ge 4$ and uniformly for all $j,k\in\Z$,
    \EQ{
 \|P_kw_*\|_{L^1_{|t|>2^{-j}}L^\I_x} \lec 2^{k}\int_{2^{-j}}^\I (2^kt)^{(1-d)/2} dt \sim 2^{(j-k)(d-3)/2}.}
In particular, we have
\EQ{
  \|P_{\ge j}w_*\|_{L^1_{|t|>2^{-j}}L^\I_x} \lec \sum_{k\ge j}  2^{(j-k)(d-3)/2} \lec 1,}
uniformly for $j\in\Z$. On the other hand, for small frequencies, we have
\EQ{
  \|P_{<j}w_*(t)\|_{L^\I_x} \lec  2^j\|w_*(t)\|_{L^{d,\I}_x}  \lec 2^j.}
We now let $\chi \in \cS(\R^d)$ with $\chi \ge 0$. As $ \chi * P_{\ge j} w_* = \cos( t D) D^{1-d} P_{\ge j} \chi$ is clearly a free wave with data in $\cS_0$, the orthogonality to $\ro$ implies that
        \EQ{ \int_\R \int_{\R^d} \rho \chi * P_{\ge j} w_* dx dt = 0 }
and hence
       \EQ{
         &\int_{|t|<2^{-j}} \int_{\R^d} \rho \chi*w_* dx dt \\
         &= \int_{|t|<2^{-j}} \rho \chi*P_{<j}w_* dx dt - \int_{|t|>2^{-j}} \rho \chi*P_{\ge j} w_* dx dt.
        }
Applying the above bounds for $w_*$, we conclude that
\EQ{
 \sup_{j\in\Z} \int_{|t|<2^{-j}}\int_{\R^d} \ro \chi*w_* dxdt \lec \|\ro\|_{L^\I_t L^1_x} \| \chi \|_{L^1_x}.}
As the integrand is non-negative, by the Monotone Convergence Theorem we see that for any $\chi \in \cS(\R^d)$ with $\chi\ge 0$ we have
   \EQ{ \int_\R \int_{\R^d} \rho \chi * w_* dx dt \lec \| \rho \|_{L^\infty_t L^1_x} \| \chi \|_{L^1_x}. }
To conclude the proof, we take $\chi$ to be an approximation to the identity, thus $\chi_\epsilon = \epsilon^{-d} \phi( \epsilon^{-1} x)$ with $\phi \ge 0$, $\int_{\R^d} \phi = 1$, and $\phi \in \cS(\R^d)$. As we have the pointwise a.e. limit $\lim_{\epsilon \to 0} (\chi_\epsilon * w_*)(t,x) = w_*(t,x)$, Fatou's lemma gives
        \EQ{ \int_\R \int_{\R^d} \rho w_* dx dt \le \liminf_{\epsilon \to 0} \int_\R \int_{\R^d} \rho \chi_\epsilon * w_* dx dt \lec \| \rho \|_{L^\infty_t L^1_x}. }
\end{proof}

\section{Extinction of the critical elements} \label{sec:proof}
The proof of Theorem \ref{thm:main} is now straightforward by a contradiction argument.
We apply Theorem \ref{thm:no time-like} to exclude time-like trajectories, and Theorem \ref{thm:orth wave} to rule out space-like trajectories.

\begin{proof}[Proof of Theorem \ref{thm:main}]
If the conclusion fails, then \cite[Theorem 1.6]{C} yields a critical element $(u,N)$ such that \eqref{eqn:pot well} holds.
Then Theorem \ref{thm:no time-like} yields $\e_0,T_0>0$ such that  we have \eqref{Lip est on c} for all $t_0,t_1\in\R$ with $|t_0-t_1|\ge T_0$.
Applying Lemma \ref{lem:transparent} to $\ro=|u|^2$ with $w=-n$, $\al=2$ and $s=0$, we see that
$|u|^2 \in \sT$ and hence Theorem \ref{thm:orth wave} implies
\EQ{ \label{eqn:u2 morawetz}
 \iint_{|x|>|t|} \tf{1}{|x|}(1-\tf{|t|}{|x|})^{(d-3)/2} |u(t,x+c(0))|^2 dxdt <\I.
}
Moreover, since $u(t, x+c(t))$ is precompact in $L^2_x$, we have $R>0$ such that
\EQ{ \label{mass bound}
  \int_{|x+c(0)-c(t)|<R} |u(t,x+c(0))|^2 dx \ge M(u)}
for all $t\in\R$.
For $t \ge \max(T_0,2R/\e_0)$ and $|x+c(0)-c(t)|<R$ we have by \eqref{Lip est on c}
\EQ{
 |x|  \in |c(t)-c(0)|+(-R,R) \pt\subset t(1+\e_0,1/\e_0)+(-R,R) \pr\subset t(1+\e_0/2,1/\e_0+\e_0/2)}
 and hence injection of \eqref{mass bound} into \eqref{eqn:u2 morawetz} implies
           \EQ{ \int_{\max(T_0,2R/\e_0)}^\I \frac{dt}{t} < \infty, }
which is a contradiction.
\end{proof}

\end{document}